\documentclass{amsart}

\usepackage{amsmath,amsthm,amssymb}

\newtheorem{theorem}{Theorem}[section]
\newtheorem{lemma}[theorem]{Lemma}

\newtheorem{definition}[theorem]{Definition}
\newtheorem{corollary}[theorem]{Corollary}
\newcommand{\ph}{\varphi}
\newcommand{\NN}{\mathbb{N}}

\newcommand{\ZZ}{\mathbb{Z}}

\newcommand{\XX}{\mathcal{X}}
\newcommand{\YY}{\mathcal{Y}}
\renewcommand{\ZZ}{\mathcal{Z}}

\title[Metastability in the Furstenberg-Zimmer tower]
{Metastability in the Furstenberg-Zimmer tower}

\author{Jeremy Avigad}
\address{Department of Philosophy and Department of Mathematical
  Sciences\\
Carnegie Mellon University\\
Pittsburgh, PA 15213}
\email{avigad@cmu.edu}
\thanks{Avigad's work has been partially supported by NSF grant
 DMS-0700174 and a grant from the John Templeton Foundation.}
\author{Henry Towsner}
\address{Department of Mathematics\\
University of California\\
Los Angeles, CA 90095-1555}
\email{htowsner@gmail.com} 
\thanks{Some of
 Towsner's work was carried out while he was a participant in the
 Semester in Ergodic Theory and Additive Combinatorics at the
 Mathematical Sciences Research Institute.}

\begin{document}

\begin{abstract}
  According to the Furstenberg-Zimmer structure theorem, every
  measure-preserving system has a maximal distal factor, and is
  weak mixing relative to that factor. Furstenberg and Katznelson used
  this structural analysis of measure-preserving systems to provide a
  perspicuous proof of Szemer\'edi's theorem. Beleznay and Foreman
  showed that, in general, the transfinite construction of the maximal
  distal factor of a separable measure-preserving system can extend
  arbitrarily far into the countable ordinals. Here we show that the
  Furstenberg-Katznelson proof does not require the full strength of
  the maximal distal factor, in the sense that the proof only depends
  on a combinatorial weakening of its properties. We show that this
  combinatorially weaker property obtains fairly low in the
  transfinite construction, namely, by the $\omega^{\omega^\omega}$th
  level.
\end{abstract}

\maketitle

\section{Introduction}

Let $\XX = (X, \mathcal B, \mu, T)$ be a measure preserving system,
that is, a finite measure space $(X, \mathcal B, \mu)$ together with a
measure-preserving transformation, $T$. A ($T$-invariant) factor $\YY$
of such a system is said to be \emph{distal} if it is the last element
of an increasing finite or transfinite sequence $(\YY_\alpha)_{\alpha
  \leq \theta}$ of factors, such that $\YY_0$ is the trivial factor,
for each $\alpha < \theta$, $\YY_{\alpha+1}$ is compact relative to
$\YY_\alpha$, and for each limit ordinal $\gamma \leq \theta$,
$\YY_{\mathcal \gamma}$ is the limit of the preceding factors. A
structural analysis due to Furstenberg and Zimmer, independently,
shows that every measure preserving system has a maximal distal
factor, and is weak mixing relative to that factor (see
\cite{furstenberg:77,furstenberg:katznelson:79,furstenberg:et:al:82}).

Furstenberg \cite{furstenberg:77} proceeded to give an
ergodic-theoretic proof of Szemer\'edi's theorem that used only a
finite sequence of compact extensions of the trivial factor. But he
noted, in passing, that one could give an alternate proof using the
maximal distal factor. Furstenberg and Katznelson
\cite{furstenberg:katznelson:79,furstenberg:81} in fact used this
strategy to prove a multidimensional generalization of Szemer\'edi's
theorem. Even for the original version of the theorem, the
Furstenberg-Katznelson proof (which draws on ideas from Ornstein, and is
presented in \cite{furstenberg:et:al:82}) is perhaps the cleanest and most
perspicuous proof of Szemer\'edi's theorem to date.

Beleznay and Foreman~\cite{beleznay:foreman:96} have shown that for
the separable spaces that arise in the proofs of Szemer\'edi's
theorem, the transfinite construction of the maximal distal factor can
extend arbitrarily far into the countable ordinals.  It is therefore
striking that the proof of a finitary combinatorial result can make
use of such a transfinite construction in an essential
way.

Our goal here is to provide a precise sense in which the
Furstenberg-Katznelson proof does not ``need'' the full transfinite
hierarchy. Specifically, we show that the argument does not require
that $\XX$ is weak mixing relative to a distal factor $\YY$; rather,
it is enough to know that $\YY$ is a limit of distal factors with respect
to which $\XX$ exhibits sufficient approximations to weak mixing
behavior. We show that such distal factors always occur fairly low
down in the transfinite hierarchy, in fact, by the $\omega^{\omega^\omega}$th
level. This helps clarify the combinatorial role of the maximal distal
factor in the Furstenberg Katznelson argument, and the axiomatic
strength needed to carry out the proof.

A central theme here is that if instead of exact limits one is
interested in having only sufficiently large pockets of approximate
stability, one can often obtain better bounds, uniformity, and/or
computability results. We referred to this phenomenon as ``local
stability'' in \cite{avigad:et:al:unp}; Tao \cite{tao:08,tao:08b} has
used the term ``metastability'' in a similar sense. In particular, we
will rely on a metastability analysis of the mean ergodic theorem due to 
Kohlenbach and Leu\c{s}tean \cite{kohlenbach:leustean:unp}.

The outline of this paper is as follows. In
Section~\ref{preliminaries:section}, we briefly outline the
Furstenberg-Katznelson proof of Szemer\'edi's theorem, introducing the
relevant definitions. In Section~\ref{results:section}, we state our
main results, which are then proved in Sections~\ref{met:section} to
\ref{all:orders:section}. In Section~\ref{logic:section}, we describe
the logical methods that underlie our work, and draw conclusions about
the axiomatic strength of the principles needed in the
Furstenberg-Katznelson proof.

We are very grateful to our anonymous referees for comments, suggestions, and
corrections, and to Ulrich Kohlenbach for helping us simplify the proofs in Section~\ref{met:section}.

\section{Preliminaries}
\label{preliminaries:section}

Szemer\'edi's Theorem states that for every $k$ and $\delta > 0$ there
is an $N$ large enough so that if $S$ is any subset of
$\{1,2,\ldots,N\}$ with density at least $\delta$, then $S$ contains
an arithmetic progression of length $k$.
Furstenberg~\cite{furstenberg:77} showed that this is equivalent to
the statement that for every measure preserving system $\XX$, every
$k$, and every set $A$ of positive measure, there is an $n$ such that
$\mu(\bigcap_{l<k}T^{-ln}A)>0$. We will henceforth refer to this
measure-theoretic equivalent as Szemer\'edi's theorem.

The $T$-invariant factors of a measure-preserving system $(X, \mathcal
B, \mu, T)$ are naturally identified with the sub-$\sigma$-algebras
$\mathcal B'$ of $\mathcal B$ that are closed under the map $A \mapsto
T^{-1}A$. It is fruitful to adopt a Hilbert-space perspective, and
consider the space $L^2(\XX)$ of square integrable functions on $\XX$,
with the isometry $\hat T$ which maps $f$ to $f \circ T$. Any
$T$-invariant factor gives rise to the $\hat T$-invariant subspace
$\mathcal Y$ of $\mathcal B'$-measurable functions of $L^2(\XX)$. This
space contains all the constant functions, and is closed under the map
$f \mapsto \max(f,0)$. Conversely, any such space gives rise to a
corresponding factor. We will henceforth use $T$ instead of $\hat T$
to denote the relevant isometry on $L^2(\XX)$, and use the term
``factor of $\XX$'' to mean a $T$-invariant subspace of $L^2(\XX)$
containing the constant functions and closed under the map $f \mapsto
\max(f,0)$. If $A$ is an element of $\mathcal B$, ``$A$ in $\YY$''
means that the characteristic function $\chi_A$ of $A$ is in $\YY$,
which amounts to saying that $A$ is in the corresponding
$\sigma$-algebra.

If $\YY$ is a factor of $\XX$, the expectation operator $E(f \mid
\YY)$ denotes the projection of $f$ onto $\YY$. More information about
factors and the expectation operator can be found, say, in
\cite{furstenberg:81}. For the most part, we will be able to restrict
our attention to the subset $L^\infty(\XX)$ of essentially bounded
elements of $L^2(\XX)$, and we will use $L^\infty(\YY)$ to denote the
essentially bounded elements of the factor $\YY$.

The Furstenberg-Zimmer structure theorem shows that any
measure-preserving system $\XX$ has a \emph{maximal distal factor},
that is, a factor $\YY$ that is built up using a transfinite sequence
of compact extensions; and that $\XX$ is weak mixing relative to
$\YY$. We now briefly review the definitions and provide a more
precise statement of the theorem.

\begin{definition}
 If $\YY$ is a factor of $\XX$, we say $\XX$ is \emph{weak mixing
   relative to $\YY$} if for every $f$ and $g$ in $L^\infty(\mathcal
 X)$,
\[
\lim_{n\rightarrow\infty}\frac{1}{n} \sum_{i<n} \int
\left[E(fT^ig\mid\YY)-E(f\mid\YY)E(T^ig\mid\YY)\right]^2 d\mu=0.
\]
\end{definition}

The following lemma presents two important consequences of relative
weak mixing. The first provides a sense in which weak mixing
extensions are also ``weak mixing of all orders.'' The second shows
that if $\XX$ is weak mixing relative to $\YY$, then $\YY$ is
``characteristic'' for the averages of the form $\frac{1}{n} \sum_{i <
  n} \prod_{l < k} T^{ln} f_l$, in the sense that only the
projections of $f_0, \ldots, f_{k-1}$ on $\YY$ bear on the limiting
behavior.

\begin{lemma}
\label{all:orders:lemma}
 Suppose $\XX$ is weak mixing relative to $\YY$. Then for every $k$
 and for all functions $f_0,\ldots,f_k$ in $L^\infty(\XX)$, the
 following hold:
\[
\lim_{n\rightarrow\infty} \frac{1}{n}\sum_{i<n} \int
 \left(E(\prod_{l<k} T^{li} f_l \mid \YY) - \prod_{l<k} T^{li} E(f_l \mid
   \YY) \right)^2 d\mu = 0.
\]
and
\[
\lim_{n\rightarrow\infty} \left\| \frac{1}{n}\sum_{i<n}
 \left(\prod_{l<k} T^{li} f_l - \prod_{l<k} T^{li} E(f_l \mid
   \YY) \right) \right\|_{L^2(\XX)} = 0.
\]
\end{lemma}

Given a factor $\YY$, write $\langle f, g \rangle_y$ for $E(fg \mid
\YY)(y)$; this provides a ``bundle'' of Hilbert spaces indexed by
elements $y$ of $\XX$ (defined up to almost everywhere equivalence).
A function $f$ in $L^2(\XX)$ is said to be \emph{almost periodic}
relative to $\YY$ if for every $\delta > 0$, there is a finite set of
functions $g_0, \ldots, g_k$ in $L^2(\XX)$ such that $\min_{i \leq k}
\| f - g_i \|_y < \delta$ for almost every $y$ in $\XX$. Another
factor $\ZZ \supseteq \YY$ is said to be a \emph{compact} extension of
$\YY$ if every element of $\ZZ$ is a limit of functions that are
almost periodic relative to $\YY$. The space $Z(\YY)$ spanned by the
functions that are almost periodic relative to $\YY$ is called the
\emph{maximal compact extension of $\YY$}.

Lemma~\ref{zy:lemma}, below, provides another characterization of
$Z(\YY)$. Given $\XX$ and a factor, $\YY$, the \emph{square of $\XX$
 relative to $\YY$}, $\XX \times_\YY \XX$, is defined in
\cite{furstenberg:77, furstenberg:81, furstenberg:katznelson:79,
 furstenberg:et:al:82}.  Here we only need the following
characterization of the Hilbert space $L^2(\XX \times_\YY
\XX)$. Start with formal elements consisting of sums $\sum_{i < n} f_i
\otimes g_i$, where $f_i$ and $g_i$ are elements of
$L^\infty(\XX)$. Define an inner product on these elements by taking
\[
\langle f \otimes g, h \otimes k \rangle_{\YY} = \langle E(fh \mid
\YY), E(g k \mid \YY) \rangle,
\]
where the right-hand side refers to the usual inner product on
$L^2(\XX)$, and extending to finite sums using bilinearity. Then
$L^2(\XX \times_\YY \XX)$ is, up to isomorphism, the completion of
this space under the associated norm. One can show that for any $h$ in
$L^\infty(\YY)$, the elements $h f \otimes g$ and $f \otimes h g$ are
identified by the norm, and so one can view $L^\infty(\YY)$ as a
embedded in $L^2(\XX \times_\YY \XX)$ via the map $h \mapsto h \otimes
1$; in particular, the real numbers are embedded as elements $c
\otimes 1$. The projection of an element $f \otimes g$ on $\YY$ is
then given by 
\[
E(f \otimes g \mid \YY) = E(f \mid \YY) E(g \mid \YY).
\]
The action of $T$ on $L^2(\XX \times_\YY \XX)$ is obtained
by taking $T(f \otimes g) = T f \otimes T g$ and extending it to the
rest of the space.

One can define multiplication by an element $f \otimes g$ by setting
$(f \otimes g) \cdot (h \otimes k) = (fh \otimes g k)$. Integration in
$L^2(\XX \times_\YY \XX)$ is given by 
\[
\int f \; d(\mu \times_\YY \mu) = \langle f, 1 \otimes 1 \rangle.
\]
In particular, if $h$ is in $L^\infty(\YY)$, 
\[
\int h \; d(\mu \times_\YY \mu) = \int h \; d\mu.
\]
There is also a lattice structure on $L^2(\XX \times_\YY \XX)$ derived
from that on $L^2(\XX)$; all we will need below is that if $f$ and $g$
are elements of $L^\infty(\XX)$, then $\| f \otimes g \|_{L^\infty(\XX
  \times_\YY \XX)} \leq \| f \|_{L^\infty(\XX)} \cdot \| g
\|_{L^\infty(\XX)}$.

If $H$ is any element of $L^\infty(\XX \times_\YY \XX)$ of the form
$\sum_{i< n} h_i \otimes g_i$ and $f$ is in $L^2(\XX)$, define
\[
H *_\YY f = \sum_{i < n} E(f h_i \mid \YY) k_i.
\]
The $*_\YY$ operation then extends to arbitrary elements of $L^2(\XX
\times_\YY \XX)$ by taking limits. For any $H$ in $L^\infty(\XX
\times_\YY \XX)$, the operation $f \mapsto H *_\YY f$ is a bounded
linear operator, with $\| H *_\YY f \|_{L^2(\XX)} \leq \| H \|_\infty
\cdot \| f \|_{L^2(\XX)}$ (see, for example, \cite[pages
130--131]{furstenberg:81}).

We will be particularly interested in elements of $L^\infty(\XX
\times_\YY \XX)$ of the form
\[
H_g^n = \frac{1}{n} \sum_{i < n} T^i (g \otimes g),
\]
where $g$ is in $L^\infty(\XX)$. The mean ergodic theorem implies that
the functions $H_g^n$ converge to a limit, $H_g$, in $L^2(\XX
\times_\YY \XX)$. For each $n$, $\| H_g^n \|_\infty$, and hence $\|
H_g \|_\infty$, is bounded by $\| g \|^2_\infty$. One can show,
moreover, that for any fixed $g$, the sequence $(H_g^n *_\YY f)$ has
a rate of convergence that depends only on a bound on $\| f
\|_\infty$. We will make use of this uniformity in
Section~\ref{weak:mixing:section}.

The following fact is established in
\cite{furstenberg:77,furstenberg:katznelson:79,furstenberg:81}, and
implicitly in \cite{furstenberg:et:al:82}:
\begin{lemma}
\label{zy:lemma}
$Z(\YY)$ is the space spanned by the set of elements of the form $H_g
*_\YY f$, as $f$ and $g$ range over $L^\infty(\XX)$.
\end{lemma}
Moreover, if $\XX$ is not weak mixing relative to $\YY$, then then
there are elements $H_g *_\YY f$ not in $\YY$. Hence:
\begin{lemma}
\label{not:weak:mixing:lemma}
If $\XX$ is not weak mixing relative to $\YY$, then $Z(\YY) \supsetneq
\YY$.
\end{lemma}
Now define $\YY_0$ to be the trivial factor, consisting of the
constant functions. By transfinite recursion, define $\YY_{\alpha + 1}
= Z(\YY_\alpha)$ for every $\alpha$, and define $\YY_{\lambda}$ to be
the factor spanned by $\bigcup_{\gamma < \lambda} \YY_\gamma$ for
every limit ordinal $\lambda$. Since $L^2(\XX)$ is separable, we have
$\YY_{\alpha + 1} = Z(\YY_\alpha) = \YY_\alpha$ at some countable
ordinal $\alpha$. By Lemma~\ref{not:weak:mixing:lemma}, $\XX$ is weak
mixing relative to $\YY$. $\YY = \YY_\alpha$ is called the
\emph{maximal distal factor}.

\begin{definition}
 Say that the factor $\YY$ is \emph{SZ} if for every $k$ and $A$ in
 $\YY$ with $\mu(A) > 0$,
\[
\liminf_{n \to \infty} \frac{1}{n}\sum_{i < n} \mu(\bigcap_{l < k}
T^{-il} A) > 0.
\]
\end{definition}
In particular, Szemer\'edi's theorem follows from the statement
``$\XX$ is SZ.'' In \cite{furstenberg:et:al:82}, this is proved as
follows:
\begin{itemize}
\item The trivial factor is SZ.
\item If a factor $\ZZ$ is SZ, so is $Z(\ZZ)$.
\item If each of a sequence $\ZZ_0, \ZZ_1, \ZZ_2, \ldots$ of factors
 is SZ, then so is the factor spanned by $\bigcup_i \ZZ_i$.
\item If a factor $\ZZ$ is SZ, and $\XX$ is weak mixing relative to
 $\ZZ$, then $\XX$ is SZ.
\end{itemize}
The first three clauses imply that the maximal distal factor, $\YY$,
is SZ. The last implies that $\XX$ is SZ, as required.

\section{Main results}
\label{results:section}

The set of countable ordinals can be given a quick inductive
definition: $0$ is a countable ordinal; if $\alpha$ is a countable
ordinal, then so is $\alpha + 1$; and if $\alpha_0, \alpha_1,
\alpha_2, \ldots$ is an increasing sequence of countable ordinals, so
is their least upper bound, which we will denote $\sup_n
\alpha_n$. Addition, multiplication, and exponentiation can be defined
recursively (see, for example, \cite{kunen:80}), and $\omega$ is
defined to be $\sup_n n$.

It is common to identify each ordinal $\alpha$ with the set $\{ \beta
\; | \; \beta < \alpha \}$ of ordinals less than it. The ordinals
serve as representatives of the order types of well-founded orderings,
which is to say, if $(X, \prec)$ is any well-founded ordering, then
$(X, \prec)$ is isomorphic to $(\alpha, <)$ for some ordinal $\alpha$.
The arithmetic operations then have natural combinatorial
interpretations.  The ordinal $\omega$ represents the order type of
the natural numbers, and $\alpha + 1$ represents the order type
obtained by appending a single element to an ordering of type
$\alpha$. The ordinal $\alpha + \beta$ represents an ordering of type
$\alpha$ followed by an order of type $\beta$. The ordinal $\alpha
\cdot \beta$ represents $\beta$ copies of an order of type $\alpha$,
that is, the order type of $\beta \times \alpha$ under lexicographic
order. The interpretation of the ordinal $\alpha^\beta$ is slightly
more complicated: it represents the set of functions from $\beta$ to
$\alpha$ that are nonzero at only finitely many arguments, where the
order is obtained by comparing the values at the largest input where
they differ. Of course, for natural numbers $n$, $\alpha^n$ can be
identified with the $n$-fold product of $\alpha$ with itself. Many
familiar properties of addition, multiplication, and exponentiation on
the natural numbers hold for the extensions to the ordinals, but not
all. For example, addition and multiplication are associative but not
commutative, since $1 + \omega = \omega$ and $2 \cdot \omega =
\omega$.

Our main theorem is that an approximation to the first property of the
maximal distal factor given in Lemma~\ref{all:orders:lemma} holds
fairly low down in the Furstenberg-Zimmer tower.

\begin{theorem}
\label{main:theorem}
 For every $k$, all functions $f_0, \ldots, f_{k-1}$ in
 $L^\infty(\XX)$, and every $\varepsilon > 0$, there are $n$ and
 $\alpha < \omega^{\omega^\omega}$ such that for every $m \geq n$,
\[
\frac{1}{m}\sum_{i<m} \int \left(E(\prod_{l<k} T^{li} f_l \mid
 \YY_\alpha) - \prod_{l<k} T^{li} E(f_l \mid \YY_\alpha) \right)^2
d\mu <
\varepsilon.
\]
\end{theorem}

In fact, our Lemma~\ref{final:lemma} proves something stronger, namely that given $f_0,
\ldots, f_{k-1}$ and $\varepsilon > 0$ there is an $n$ with
``many'' such $\alpha < \omega^{\omega^\omega}$, in an appropriate
combinatorial sense. We obtain the following as a consequence of this stronger fact:

\begin{corollary}
\label{main:corollary}
 For every $k$, all functions $f_0, \ldots, f_{k-1}$ in
 $L^\infty(\XX)$, and every $\varepsilon > 0$, there are $n$ and
 $\alpha < \omega^{\omega^\omega}$ such that for every $m \geq n$,
\[
\left\|
 \frac{1}{m}\sum_{i<m}
 \left(\prod_{l<k} T^{ln} f_l - \prod_{l<k} T^{ln} E(f_l \mid
   \YY_\alpha) \right) \right\|_{L^2(\XX)} < \varepsilon.
\]
\end{corollary}

We emphasize that although Theorem~\ref{main:theorem} is new,
Corollary~\ref{main:corollary} is not: using an altogether different
argument, Furstenberg~\cite{furstenberg:77} showed that for each $k$,
$\YY_k$ is characteristic for the averages with $k$-fold products. Our
methods are quite general, however, and work in other
situations involving transfinite constructions of factors; see \cite{towsner:draft}. Moreover,
our argument provides some insight into the role of the maximal distal
factor in the Furstenberg-Katznelson argument, providing a general
explanation as to why the full strength of the construction is not
needed to obtain the combinatorial result.

It is worth noting that for $k = 2$, Theorem~\ref{main:theorem}
describes a weaker version of relative weak mixing. In that case, the
discussion at the end of Section~\ref{weak:mixing:section} shows that the
theorem holds with $\omega$ in place of $\omega^{\omega^\omega}$. It is not hard
show that here $\omega$ cannot be replaced by any finite ordinal $K$. Otherwise,
fixing $f_0 = f_1 = f$, we would have that for every $\varepsilon > 0$ there is an $\alpha < K$ such 
that the conclusion of the theorem holds. By the pigeonhole principle, this would imply 
that there is a single $\alpha < K$ that works for every $\varepsilon$, which 
is to say, $f$ is weak mixing relative to $\YY_\alpha$. But, by the results of Beleznay
and Foreman \cite{beleznay:foreman:96}, there are measure preserving systems with 
functions $f$ that are not weak mixing relative to any finite level of the Furstenberg-Zimmer 
hierarchy. So, for such functions, the least $\alpha$ satisfying the conclusion 
of Theorem~\ref{main:theorem} must approach $\omega$ as $\varepsilon$ approaches $0$.
Our proof gives an explicit bound on $\alpha$ depending on $k$ and $\varepsilon$; we do not 
know the extent to which that bound is sharp. 

For $k > 2$, the statement of Lemma~\ref{final:lemma} gives slightly more information, 
in terms of a bound less than $\omega^{\omega^\omega}$ depending on $k$. But, once again, we do not know the extent to which this bound is sharp, nor even that a bound of $\omega$ itself is insufficient.

Note that our corollary is even weaker than saying that some
$\YY_\alpha$, with $\alpha < \omega^{\omega^\omega}$, is characteristic for the
limit in question. But, as we now show, once we know that $\YY_\alpha$
is SZ for each $\alpha $ less than or equal to $\omega^{\omega^\omega}$, this
strictly weaker property is sufficient to obtain Szemer\'edi's
theorem. In fact, the proof is only a slight modification of the usual
Furstenberg-Katznelson argument, e.g.~\cite[Theorem
8.3]{furstenberg:et:al:82}.

\begin{theorem}
$\XX$ is SZ.
\end{theorem}

\begin{proof}
 Suppose we are given a set $A$ in $\mathcal B$ such that $\mu(A) >
 0$. Since
\[
\frac{1}{n}\sum_{i < n} \mu(\bigcap_{l = 0}^k
T^{-il} A) =
\frac{1}{n}\sum_{i < n} \int \prod_{l < k}
T^{il} \chi_A d\mu,
\]
 our goal is to show that there is a $\delta$ such that the
 right-hand side is greater than $\delta$ for sufficiently large
 $n$.

 For each $j$, let $\alpha_j$ be the least ordinal such that for
 sufficiently large $n$,
\[
\left\| \frac{1}{n}\sum_{i<n} \left(\prod_{l=0}^k T^{il} \chi_A -
   \prod_{l=0}^k T^{il} E(\chi_A \mid \YY_{\alpha_j}) \right)
\right\|_{L^2(\XX)} < 1 / j.
\]
Set $\alpha = \sup \alpha_j \leq \omega^{\omega^\omega}$, so that
$\YY_\alpha$ is the
factor spanned by $\bigcup_j \YY_{\alpha_j}$.

Since $\chi_A$ is nonnegative, so is $E(\chi_A \mid \YY_\alpha)$. Let
\[
B = \{x\mid E(\chi_A\mid\YY_\alpha)(x)\geq\mu(A)/2\}.
\]
 Since
\[
\mu(A)=\int_B E(\chi_A\mid\YY_\alpha)
d\mu+\int_{\overline{B}}E(\chi_A\mid\YY_\alpha)d\mu\leq\mu(B)+\mu(A)/2,
\]
it follows that $\mu(B) \geq \mu(A)/2$. Since $\YY_\alpha$ is SZ,
there is a $\delta$ such that
\[
\frac{1}{n}\sum_{i < n} \mu(\bigcap_{l = 0}^k
T^{-il} B) > \delta
\]
whenever $n$ is sufficiently large.

For each $j$, set
\[
B_j = \{ x \in B \mid E(\chi_A \mid \YY_{\alpha_j})(x) > \mu(A) / 4 \}.
\]
% Then, by Chebyshev's theorem, we have
% \begin{align*}
% \mu(B - B_j) & \leq \mu(\{ x \mid E(\chi_A \mid \YY_\alpha)(x) -
%  E(\chi_A \mid \YY_{\alpha_j})(x) \geq \mu(A) / 4 \}) \\
% & \leq \frac{16 \| E(\chi_A \mid \YY_\alpha) - E(\chi_A \mid \YY_{\alpha_j})
% \|^2}{\mu(A)^2}.
% \end{align*}
Since $\YY_\alpha$ is the limit of the factors $\YY_{\alpha_j}$, we
can make $\mu(B - B_j)$ as small as we want by making $j$
sufficiently large. We will choose $j$ large enough so that $\mu(B -
B_j) < \delta / (2 k)$, so that for any $i$ we have
\begin{align*}
\mu(\bigcap_{l < k}
 T^{-il} B_j) & \geq \mu(\bigcap_{l <
   k} T^{-il} B) - k \cdot (\delta / (2k)) \\
 & = \mu(\bigcap_{l <
   k} T^{-il} B) - \delta / 2.
\end{align*}
Then, since $E(\chi_A \mid \YY_{\alpha_j})
\geq \frac{\mu(A)}{4} \chi_{B_j}$, we will have
\begin{align*}
 \frac{1}{n} \sum_{i < n} \int \prod_{l < k} T^{il} E(\chi_A \mid
 \YY_{\alpha_j}) d \mu & \geq \frac{\mu(A)^k}{4^k} \frac{1}{n} \sum_{i <
   n} \int \prod_{l < k} T^{il} \chi_{B_j}
 d\mu \\
 & = \frac{\mu(A)^k}{4^k} \frac{1}{n} \sum_{i < n} \mu(\bigcap_{l < k}
 T^{-il} B_j) \\
 & \geq \frac{\mu(A)^k}{4^k} \frac{1}{n} \sum_{i < n} (\mu(\bigcap_{l <
   k} T^{-il} B) - \delta / 2) \\
 & \geq \frac{\mu(A)^k}{4^k} (\delta - \delta / 2) \\
 & = \frac{\mu(A) \cdot \delta}{2^{2k + 1}}
\end{align*}
for sufficiently large $n$. Call the right-hand side $\eta$.

Choose $j$ so that in addition to satisfying $\mu(B - B_j) < \delta /
(2 k)$, we also have $1 / j < \eta / 2$. Then, by the construction of
the sequence $(\alpha_j)$, we have
\begin{align*}
\frac{1}{n}\sum_{i < n} \int \prod_{l < k}
T^{il} \chi_A d\mu & \geq \frac{1}{n} \sum_{i < n} \int \prod_{l < k} T^{il}
E(\chi_A \mid \YY_{\alpha_j}) - \eta / 2 \\
& \geq \eta / 2,
\end{align*}
for sufficiently large $n$, as required.
\end{proof}

We now turn to the proof of Theorem~\ref{main:theorem}. Our proof
tracks the usual proof that $\XX$ is weak mixing of all orders
relative to the maximal distal factor, $\YY$; but wherever that proof
asserts that $\XX$ exhibits some behavior relative to $\YY$, we assert
instead that $\XX$ exhibits some approximation to that behavior,
relative to sufficiently many $\YY_\alpha$. The following definitions
provide the notions of ``sufficiently many'' that we will need. If
$\theta$ and $\eta$ are ordinals, $(\theta,\eta]$ denotes the interval
$\{ \delta \; | \; \theta < \delta \leq \eta \}$.

\begin{definition}
 If $\alpha$ is an ordinal, say $s$ is an \emph{$\alpha$-sequence} if
 $s = (s_\beta)_{\beta \leq \alpha}$ is a strictly increasing
 sequence of ordinals indexed by ordinals less than or equal to
 $\alpha$. Say $t$ is a \emph{$\beta$-subsequence of $s$} if $t$ is a
 $\beta$-sequence and a subsequence of $s$. If $s$ is an
 $\alpha$-sequence, the \emph{span} of $s$, written
 $\mathrm{span}(s)$, is $(s_0,s_\alpha]$.
\end{definition}

\begin{definition}
 If $s$ is an $\alpha$-sequence and $P(\delta)$ is any property, say
 \emph{$P$ holds for $s$-many $\delta$} if for every $\beta <
 \alpha$, there is a $\delta$ in $(s_{\beta},s_{\beta+1}]$ such that
 $P(\delta)$ holds.
\end{definition}

In other words, $P(\delta)$ holds for $s$-many $\delta$ if, roughly
speaking, there is an element satisfying $P$ between any two
consecutive elements of $s$.

\section{Approximating the mean ergodic theorem}
\label{met:section}

Let $\mathcal H$ be any Hilbert space, $T$ an isometry, and $f$ any
element of $\mathcal H$. For every $n \geq 1$, let $A_n f = (1 / n)
\sum_{i < n} T^i f$. The mean ergodic theorem says that the sequence
$(A_n f)$ converges in the Hilbert space norm; in other words, for
every $\varepsilon > 0$, there is an $n$ such that for every $m \geq
n$ we have $\| A_m f - A_n f \| < \varepsilon$.

Now let $(\mathcal H_\alpha)_{\alpha \in S}$ be a sequence of Hilbert
spaces indexed by ordinals in some set $S$, let $(T_\alpha)$ be a
sequence of isometries, and let $(f_\alpha)$ be a sequence of
elements. Given $\varepsilon > 0$, the mean ergodic theorem implies
that for every $\alpha$ there is an $n$ as above, but, of course,
different $\alpha$'s may call for different $n$'s.

Here we will be concerned with the case where the spaces $\mathcal
H_\alpha$ are the ones denoted by $L^2(\XX \times_{\YY_\alpha} \XX)$
in Section~\ref{preliminaries:section}, and for some $L^\infty(\XX)$
function $f$, each $f_\alpha$ is the element $f \otimes f$ in the
corresponding space. Our goal is to obtain for every $\varepsilon > 0$
a single $n$ that works for sufficiently many $\alpha$'s. In
Section~\ref{weak:mixing:section}, we will use this to show that
approximate weak mixing behavior occurs sufficiently often relative to
the factors $\YY_\alpha$.

Our original presentation relied on information extracted in \cite{avigad:et:al:unp} from the proof of the mean ergodic theorem due to Riesz \cite{riesz:41}. We are grateful to Ulrich Kohlenbach for pointing out the proofs of the results in this section could be simplified considerably by using information extracted by Kohlenbach and Leu\c{s}tean \cite{kohlenbach:leustean:unp} from a proof of the mean ergodic theorem by Garrett Birkhoff \cite{birkhoff:39}. The following lemma is implicit in \cite{kohlenbach:leustean:unp}, and holds more generally for nonexpansive mappings on a uniformly convex Banach space. It says, roughly, that from a bound on $k$ such that $\| A_k f \|$ is close to its infimum, one can determine a value $n$ beyond which the sequence of ergodic averages is close to its limit.

\begin{lemma}
\label{kl:lemma}
For every $B$ and $\varepsilon > 0$ there is a $\gamma > 0$ with the following property: for every $i$ there is an $n$ such that if $f$ is any element of a Hilbert space $\mathcal H$ with $\| f \| \leq B$, $T$ is an isometry, and there is a $k \leq i$ such that
\begin{equation}
\label{kl:eq}
\| A_k f \| \leq \| A_j f \| + \gamma
\end{equation}
holds for every $j$, then 
\[
\| A_n f - A_m f \| < \varepsilon
\]
for every $m \geq n$.
\end{lemma}

\begin{proof}
Using the notation of \cite{kohlenbach:leustean:unp}, let $M = 16 B / \varepsilon$, let $n = Mi$, and let $\gamma = (\varepsilon / 16) \eta(\varepsilon / 8b)$, where $\eta$ is a modulus of convexity for Hilbert space. The proof in 
\cite[Section 4, pages 1913--1914]{kohlenbach:leustean:unp} shows that if (\ref{kl:eq}) holds for every $j$, then $\| A_m f - A_l f \| < \varepsilon$ holds for every $m$ and $l$ greater than or equal to $n$. (The $N$ in \cite{kohlenbach:leustean:unp} plays the role of our $i$, and $P$ corresponds to our $n$. Because we are assuming that (\ref{kl:eq}) holds for all $j$, the conclusion of the argument in \cite{kohlenbach:leustean:unp} holds for arbitrary functions $g$.)
\end{proof}

We now fix a sequence of Hilbert spaces $(\mathcal H_\alpha)_{\alpha
 \in S}$, where $S$ is some set of ordinals and each $\mathcal
H_\alpha$ comes equipped with its own inner product $\langle \cdot,
\cdot \rangle_\alpha$ and norm $\| \cdot \|_\alpha$. We also fix an
isometry $T_\alpha$ on each $H_\alpha$. The next theorem deals
with sequences $(f_\alpha)_{\alpha \in S}$, where each
$f_\alpha$ is in $H_\alpha$. For readability, we will adopt the
practice of dropping the subscripted $\alpha$ on terms like $f_\alpha$
and $T_\alpha$ when the context makes it clear. Thus, for example, the expression $\| A_n f \|_\alpha$ really means $\| A_n f_\alpha \|_\alpha$.

Although the sequences $(A_n f)$ converge in each $\mathcal H_\alpha$, they may have very different rates of convergence. The next lemma shows that, nonetheless, as long as there is a uniform bound on the values $\| f \|_\alpha$, for any $\varepsilon > 0$ there is always an $n$ large enough so that, for ``many'' $\alpha$'s, $\| A_n f - A_m f \| < \varepsilon$ holds for all $m \geq n$.

\begin{theorem}
\label{second:lemma}
Let $\varepsilon > 0$ and $B > 0$. Then there is a natural number $K$
such that for every $\alpha^K$-sequence $s$ and every sequence of
elements $(f_\delta)_{\delta \in \mathrm{span}(s)}$ bounded by $B$ in
norm, there are a natural number $n$ and an $\alpha$-subsequence $t$
of $s$, such that the property
\begin{quote}
$\| A_n f - A_m f \|_\delta < \varepsilon$ for every $m \geq n$
\end{quote}
holds for $t$-many $\delta$.
\end{theorem}

\begin{proof}
For each $i$, write $a_{i,\delta} = \inf_{j \leq i} \| A_j f_\delta \|_\delta$. According to the convention above, we will leave the subscripted $\delta$'s off of $f_\delta$ and $a_{i,\delta}$ but keep the dependence in mind. For each $\delta$, the sequence $a_i$ is a decreasing sequence bounded above by $B$ and below by $0$. Let $\gamma$ be as guaranteed to exist by Lemma~\ref{kl:lemma}. 

Now let $K = \lceil B / \gamma \rceil + 1$, let $s$ be any $\alpha^K$-sequence, and let $(f_\delta)_{\delta \in \mathrm{span}(s)}$ be a sequence of elements bounded by $B$ in
norm. It suffices to show that there are a natural number $i$ and an $\alpha$-subsequence $t$ of $s$ such that the property
\begin{quote}
for every $j > i$, $a_i \leq a_j + \gamma$
\end{quote}
holds for $t$-many $\delta$, because then the hypotheses of Lemma~\ref{kl:lemma}, and hence the conclusion, are satisfied for these $\delta$'s.

Suppose otherwise. Then we have the following (*):
\begin{quote}
 For every $i$ and $\alpha$-subsequence $t$ of $s$, there are $j > i$
 and $\beta < \alpha$ such that for every $\delta \in
 (s_\beta,s_{\beta+1}]$, $a_j < a_i - \gamma$.
\end{quote}
Start with $i_0 = 0$, in which case $a_{i_0} = \| f \|$. Think of $s$ as consisting of $\alpha$-many consecutive
$\alpha^{K-1}$-subsequences, overlapping only at the endpoints, so that
the last element of one is the first element of the next. We can then
use (*) to find an $i_1 > i_0$ and one of those subsequences such that $a_{i_1} < a_{i_0} - \gamma$ on its span. Then think of
\emph{that} subsequence as consisting of $\alpha$-many consecutive
$\alpha^{K-2}$-subsequences, and use (*) again to find an $i_2 > i_1$ and
one of \emph{those} sequences such that $a_{i_2} < a_{i_1} - \gamma$ on its span.
Continuing in this way we ultimately find a $\delta$
and a sequence $a_{i_0}, a_{i_1}, \ldots, a_{i_K}$ such that for each $u < K$ we have $a_{i_{u+1}} < a_{i_u} - \gamma$ at $\delta$. But this contradicts the fact that, by the choice of $K$,  $a_{i_u}$ can decrease by $\gamma$ at most $K$ times.
\end{proof}

We now specialize to the situation where each $\mathcal H_\alpha$ is
$L^2(\XX \times_{\YY_\alpha} \XX)$, and each $f_\alpha$ is $f \otimes
f$, for some fixed $L^\infty(\XX)$ function $f$. This meets
the requirements of the lemma, because we have
$\| f \otimes f \|^2_\alpha = \langle f \otimes f, f \otimes f
\rangle_\alpha = \int E(f^2 \mid \YY_\alpha)^2 d\mu \leq \| f
\|^4_\infty$ for each $\alpha$.
Thus we have a uniform approximate
version of the mean ergodic theorem for $L^2(\XX \times_{\YY_\alpha}
\XX)$.

\begin{theorem}
\label{met:theorem}
Let $\varepsilon > 0$ and $B > 0$. Then there is a natural number $K$
such that for every $\alpha^K$-sequence $s$ and every $f$ in
$L^\infty(\XX)$ with $\| f \|_\infty \leq B$, there are a natural
number $n$ and an $\alpha$-subsequence $t$ of $s$, such that the
property
\begin{quote}
for every $m \geq n$, $\| A_n (f \otimes f) - A_m (f \otimes
f) \|_\delta < \varepsilon$
\end{quote}
holds for $t$-many $\delta$.
\end{theorem}

Notice that if $s$ is the trivial $1$-sequence $\delta,\delta+1$,
Theorem~\ref{met:theorem} simply asserts that $A_n (f \otimes f)$
converges in $\XX \times_{\YY_{\delta + 1}} \XX$.

\section{Approximating weak mixing}
\label{weak:mixing:section}

Let $g$ be in $L^\infty(\XX)$. Now notice that the elements $H_g^n$ of
the spaces $L^2(\XX \times_{\YY_\delta} \XX)$, defined in
Section~\ref{preliminaries:section}, are none other than the elements
$A_n (g \otimes g)$, where $A_n$ is as in Section~\ref{met:section}.
Let $f$ be any element of $L^2(\XX)$. As we observed in
Section~\ref{preliminaries:section}, the rate of convergence of $H_g^n
*_{\YY_\delta} f$ to $H_g f$ in $L^2(\XX \times_{\YY_\delta} \XX)$
depends only on the rate of convergence of $H_g^n$ to $H_g$ and on $\|
f \|_{L^2(\XX)}$. 

% To simplify the statements of the lemmas and
% theorems that follow, we restrict our attention to functions $f$ and
% $g$ with $\| f \|_\infty \leq 1$ and $\| g \|_\infty \leq 1$. Thus we
% have the following corollary to Theorem~\ref{met:theorem}.

% \begin{lemma}
% \label{wm:lemma}
% Let $\varepsilon > 0$ and $B>0$. Then there is a natural
% number $K$ such that for every
% $\alpha^K$-sequence $s$ and every $g$ in $L^\infty(\XX)$ with $\| g
% \|_\infty \leq B$, there are a natural number $n$ and an
% $\alpha$-subsequence $t$ of $s$, such that the property
% \begin{quote}
%  for every $m \geq n$ and $f$ with $\| f \|_{L^2(\XX)} \leq 1$, $\|
%  H^m_g *_{\YY_\delta} f - H_g *_{\YY_\delta} f \| < \varepsilon$
% \end{quote}
% holds for $t$-many $\delta$.
% \end{lemma}

We now use this to obtain our first main result, to the effect that
$\XX$ exhibits approximate weak mixing behavior relative to the
factors $\YY_\delta$, for sufficiently many ordinals $\delta$.

\begin{theorem}
\label{wm:theorem}
For every $\varepsilon > 0$ and $B>0$ there is a natural number $K$ such that
for every $\alpha \geq \omega$, every $\alpha^K$-sequence $s$, and
every $f$ and $g$ with $\| f \|_\infty \leq B$ and $\| g \|_\infty
\leq B$, there are an $n$ and an $\alpha$-subsequence $t$ of $s$, such
that the property
\begin{quote}
 for every $m \geq n$, $\frac{1}{m} \sum_{i < m}\int\left[E(fT^ig \mid \YY_\delta)- E(f \mid
   \YY_\delta)T^iE(g \mid \YY_\delta) \right]^2 d\mu<\varepsilon$
\end{quote}
 holds for $t$-many $\delta$.
\end{theorem}

\begin{proof}
 For any $\delta$, if we set $h_\delta$ equal to $f-E(f\mid\delta)$,
 we have
\begin{align*}
\frac{1}{m}\sum_{i<m}\int & \left[E(fT^ig\mid \YY_\delta)-E(f\mid
\YY_\delta)T^iE(g\mid \YY_\delta)\right]^2d\mu\\
& =\frac{1}{m}\sum_{i<m}\int\left[E(h_\delta T^ig+E(f\mid
\YY_\delta)T^ig\mid \YY_\delta)-E(f\mid \YY_\delta)T^iE(g\mid
\YY_\delta)\right]^2d\mu\\
& =\frac{1}{m}\sum_{i<m}\int\left[E(h_\delta T^ig\mid \YY_\delta)\right]^2d\mu\\
& =\frac{1}{m}\sum_{i<m}\int E(h_\delta T^ig\mid \YY_\delta)E(h_\delta
T^ig\mid \YY_\delta) \; d\mu\\
& =\int E(h_\delta\frac{1}{m}\sum_{i<m}T^igE(h_\delta T^ig\mid
\YY_\delta)\mid \YY_\delta) \; d\mu\\
& =\int E(h_\delta \cdot (H^m_g *_{\YY_\delta} h_\delta) \mid
\YY_\delta) \; d\mu \\
& =\int h_\delta \cdot (H^m_g *_{\YY_\delta} h_\delta) \; d\mu.
\end{align*}
Here is the idea: by Theorem~\ref{met:theorem}, we can make $H^m_g
*_{\YY_\delta} h_\delta$ close to $H_g *_{\YY_\delta} h_\delta$ for
sufficiently many $\delta$. By the definition of the transfinite
sequence of factors ($\YY_\delta$), $H_g *_{\YY_\delta} h_\delta$ is
in $\YY_{\delta + 1}$. On the other hand, $h_{\delta+1}$ is orthogonal
to $\YY_{\delta + 1}$, so $\int h_{\delta+1} \cdot (H_g *_{\YY_\delta}
h_\delta) \; d\mu$ is equal to $0$. Thus, as long as
\[
h_{\delta+1} - h_{\delta } = E(f \mid \YY_{\delta + 1}) - E(f \mid \YY_\delta)
\]
is small, $\int h_\delta \cdot (H^m_g *_{\YY_\delta} h_\delta) \;
d\mu$ will be close to $0$, as required.

But now suppose we obtain a countable sequence $\delta_0 < \delta_1 <
\delta_2 < \ldots$ of ordinals, where $H^m_g *_{\YY_{\delta_i}}
h_{\delta_i}$ is close to $H_g *_{\YY_{\delta_i}} h_{\delta_i}$ for
each $i$. Then since $(E(f \mid \YY_{\delta_i}))_{i \in \NN}$ is a
sequence of projections of $f$ onto increasing factors, for some $i$
we will have that $E(f \mid \YY_{\delta_{i+1}}) - E(f \mid
\YY_{\delta_i})$, and hence $h_\delta - h_{\delta+1}$, is sufficiently
small. Such a $\delta_i$ is then one of the ordinals we are after.

The details are as follows. Given $\varepsilon > 0$, apply
Lemma~\ref{met:theorem} to $\varepsilon / 2B$, and
let $K$ satisfy the conclusion of that lemma. We claim that $2 K$
satisfies the conclusion of Theorem~\ref{wm:theorem}.

Suppose we are given an $\alpha^{2 K}$-sequence $s$, and $f$ and $g$
satisfying $\| f \|_\infty \leq B$ and $\| g \|_\infty \leq B$. Since
$\alpha \geq \omega$, we have $\alpha^{2K} = (\alpha^2)^K \geq (\omega
\cdot \alpha)^K$, and we can restrict our attention to the initial
$(\omega \cdot \alpha)^K$-subsequence of $s$. By our choice of $K$,
there is an $\omega \cdot \alpha$-subsequence $t$ such that the
property (*)
\begin{quote}
for every $m \geq n$ and $h$ with $\| h \|_{L^2(\XX)} \leq B$,
$\| H^m_g *_{\YY_\delta} h -
H_g *_{\YY_\delta} h \| < \varepsilon / 2$
\end{quote}
holds for $t$-many $\delta$.

Let $t'$ be the $\alpha$-sequence obtained by taking every $\omega$th
element of $t$; that is, That is, define $t'_{\beta} = t_{\omega \cdot
 \beta}$ for each $\beta \leq \alpha$. We claim that the property (**)
\begin{quote}
 for every $m \geq n$, $\int h_\delta \cdot (H^m_g *_{\YY_\delta}
 h_\delta) \; d\mu < \varepsilon$
\end{quote}
holds for $t'$-many $\delta$, as required.

To prove this, let $\beta < \alpha$. We need to show that there is a
$\delta$ satisfying
\[
t_{\omega \cdot \beta} = t'_\beta < \delta \leq t'_{\beta + 1} =
t_{\omega \cdot \beta + \omega}
\]
such that $\int h_\delta \cdot (H^m_g *_{\YY_\delta} h_\delta) \; d\mu
< \varepsilon$. By our choice of $t$, for every $i$ there is a
$\delta_i \in (t_{\omega \cdot \beta + i}, t_{\omega \cdot \beta + i +
 1}]$ satisfying (*) with $\delta_i$ in place of $\delta$. Choose $i$
such that
\[
\| h_{\delta_i + 1} - h_{\delta_i} \| =
\| E(f \mid \YY_{\delta_i+1}) - E(f \mid \YY_{\delta_i}) \| <
\varepsilon / 2B^2.
\]

Now for $\delta = \delta_i$, we have
\begin{align*}
h_\delta \cdot (H^m_g *_{\YY_\delta} h_\delta) =
h_\delta \cdot & ((H^m_g *_{\YY_\delta} h_\delta) - (H_g *_{\YY_\delta}
h_\delta)) + \\
& (h_\delta - h_{\delta+1}) \cdot (H_g *_{\YY_\delta}
h_{\delta+1}) + h_{\delta+1} \cdot (H_g *_{\YY_\delta} h_\delta).
\end{align*}
For every $m \geq n$, by (*), the first term is bounded in $L^2(\XX)$
norm by $\| h_\delta \|_\infty \cdot \varepsilon / 2B$, which is less than
$\varepsilon / 2$, since since $\| h_\delta \|_\infty \leq B$. The second
term is bounded in $L^2(\XX)$ norm by $(\varepsilon / 2B^2) \cdot \| H_g
*_{\YY_\delta} h_{\delta+1} \|_\infty$, which is less than
$\varepsilon / 2$, since $\| H_g \|_\infty \leq B^2$. The integral of the
last term is $0$, since $h_{\delta + 1}$ is orthogonal to
$\YY_{\delta+1}$ and $H_g *_{\YY_\delta} h_\delta$ is an element of
$\YY_{\delta+1}$. Hence we have $\int h_\delta \cdot (H^m_g
*_{\YY_\delta} h_\delta) \; d\mu < \varepsilon$, as required.
\end{proof}

Notice that, in the previous proof, we did not really need an $(\omega
\cdot \alpha)$-sequence $t$ satisfying (*); an $(L \cdot
\alpha)$-sequence would have been sufficient, with $L > 4 /
\varepsilon^2$. Furthermore, if $\alpha$ is any limit ordinal, then $L
\cdot \alpha = \alpha$. Note also that we could just as well have
switched the two steps of thinning $s$: starting with an $(\alpha^K
\cdot L)$-sequence $s$, we could have obtained an
$\alpha^K$-subsequence $t'$ such that $\| E(f | \YY_\gamma) - E(f |
\YY_\delta) \| < \varepsilon / 2$ for every $\gamma$ and $\delta$ in the
span of $t'$, and then applied Lemma~\ref{met:theorem} to obtain an
$\alpha$-subsequence $t$ such that (*) holds for $t$-many $\delta$. In
particular, any sequence of length $L$ is sufficient to obtain a
$1$-sequence $t$ such that the conclusion of Theorem~\ref{wm:theorem}
holds for $t$-many $\delta$, which is to say, at least one $\delta$ in
the span of $t$. This shows that for $k = 2$,
Theorem~\ref{main:theorem} holds with $\omega$ in place of
$\omega^{\omega^\omega}$.

\section{Approximating weak mixing of all orders}
\label{all:orders:section}

In this section, we show how to approximate the property of being weak
mixing of all orders relative to the maximal distal factor below level
$\omega^{\omega^\omega}$ in the Furstenberg-Zimmer tower. Our proof
parallels the proof in \cite{furstenberg:et:al:82} that the fact that
$\XX$ is weak mixing relative to $\YY$ implies that it is weak mixing
of all orders relative to $\YY$; but wherever that proof asserts that
some property holds relative to $\YY$, we assert that a corresponding
property holds relative to $\YY_\delta$, for sufficiently many
$\delta$'s. Unlike the properties in the previous section, for which
sequences of length $\alpha^K$ with integer $K$ were sufficient, we
will need to consider sequences of the length $\alpha^\theta$, where
$\theta$ is ordinal less than $\omega^\omega$.

We start by proving three technical lemmas, which correspond to claims
that are trivial in the original proof, but become more complicated in
our modified version. To give a typical example, if both
\[
\frac{1}{m}\sum_{i<m}
 \left \| E(fT^ig \mid \YY)-E(f \mid \YY)T^iE(g \mid
   \YY) \right\| \rightarrow 0\]
and
\[\frac{1}{m}\sum_{i<m}
 \left\| E(f'T^ig \mid \YY)-E(f' \mid \YY)T^iE(g \mid
   \YY) \right\| \rightarrow 0,\]
then
\[\frac{1}{m}\sum_{i<m}
\\\left\| E((f+f')T^ig \mid \YY)-E((f+f') \mid \YY)T^iE(g \mid \YY)
\right\| \rightarrow 0,\] and such inferences are used many times in
the proof in \cite{furstenberg:et:al:82}. In our ``approximate''
version, however, we typically wish to show that for each
$\varepsilon$ we can find ``many'' $\delta$ such that the third
average is less than $\varepsilon$ with respect to $\YY_\delta$, using
the fact that the first two averages are small with respect to many
$\YY_\delta$.  In particular, this requires finding many $\delta$ such
that the first two averages are small simultaneously at $\YY_\delta$.

Since the same situation recurs during the proof with many different
choices of the precise averages being controlled, we will state the
lemmas in a very general form.  We will work with properties
$\ph(\delta)$ which assert that a quantity computed with respect to
$\YY_\delta$ is small; for instance, in the example
above, the first choice of $\varphi(f,m,\delta)$ would be
\[
\frac{1}{m}\sum_{i<m}\left\|E(fT^ig \mid \YY_\delta)-E(f \mid
  \YY_\delta)T^iE(g \mid \YY_\delta) \right\| \leq\varepsilon.
\]
We will use the fact that such properties are continuous in the
following sense.

\begin{definition}
  A property $\varphi(\vec x,\delta)$ is \emph{continuous in $\delta$}
  if for any choice of values $\vec t$ for $\vec x$ such that
  $\varphi(\vec t,\delta_i)$ holds for all $i$, also $\varphi(\vec
  t,\sup_i\delta_i)$.
\end{definition}

The first lemma says that we can arrange for a pair of continuous properties to
hold for many $\delta$ simultaneously by arranging for each property, in turn,
to hold sufficiently often.

\begin{lemma}
 Suppose $\ph_1(\vec x,\delta)$ and $\ph_2(\vec x,\delta)$ are continuous in
$\delta$. Fix $\vec x$. 

Suppose there is a $\theta_1<\omega^{p}$ such that for every
$\alpha^{\theta_1}$-sequence $s$ with $\alpha\geq\omega$ and every $f$ with
$\|f\|_{L^\infty}\leq B$, there are a natural number $n_1$ and an
$\alpha$-subsequence $t$ of $s$ such that the property
\begin{quote}
for every  $m\geq n_1$, $\varphi_1(f,m,\delta)$
\end{quote}
holds for $t$-many $\delta$.

Suppose that, additionally, there is a $\theta_2<\omega^{q}$ such that for
every $\alpha^{\theta_2}$-sequence $s$ with $\alpha\geq\omega$ and every
$f$ with $\|f\|_{L^\infty}\leq B$, there are a natural number $n_2$
and an $\alpha$-subsequence $t$ of $s$ such that the property
\begin{quote}
for every  $m\geq n_2$, $\varphi_2(f,m,\delta)$
\end{quote}
holds for $t$-many $\delta$.

Then there is a
$\theta<\omega^{p+q-1}$ such that for every $\alpha^\theta$-sequence $s$ with
$\alpha\geq\omega$ and every $f$ with $\|f\|_{L^\infty}\leq B$, there
are a natural number $n$ and an $\alpha$-subsequence $t$ of $s$ such
that the property
\begin{quote}
for every  $m\geq n$, $\varphi_1(f,m,\delta)$ and $\varphi_2(f,m,\delta)$
\end{quote}
holds for $t$-many $\delta$.
\label{tech:pair}
\end{lemma}
\begin{proof}
 Given $\theta_1$ and $\theta_2$ as in the hypotheses, let $\theta
 =2\cdot \theta_1\cdot \theta_2$.  Let $s$ be an $\alpha^{2\cdot
   \theta_1\cdot \theta_2}$-sequence, and let $f$ be given.  Applying
 the hypotheses sequentially, we obtain an $\alpha^{2}$-subsequence
 $t'$ and an $n = \max (n_1,n_2)$ such that both the properties
 $\forall m\geq n \; \varphi_1(f,m,\delta)$ and $\forall m\geq n \;
 \varphi_2(f,m,\delta)$ hold for $t'$-many $\delta$.  Since
 $\alpha\geq\omega$, we can consider the $\alpha$-subsequence $t$ of
 $t'$ given by setting $t_\beta:=t'_{\beta\cdot\omega}$ for each
 $\beta \leq \alpha$.  For any $\beta<\alpha$ and any $n<\omega$,
 there is a $\delta$ in $(t'_{\beta\cdot \omega+n},t'_{\beta\cdot
   \omega+n+1}]$ such that $\forall m\geq n \; \varphi_1(f,m,\delta)$
 holds, so ordinals with this property occur unboundedly below
 $t_{\beta+1}=t'_{(\beta+1)\cdot\omega}$.  In particular, $\forall
 m\geq n \; \varphi_1(f,m,t_{(\beta+1)\cdot\omega})$ and similarly for
 $\varphi_2$, so the sequence $t$ witnesses the lemma.
\end{proof}

We will often want to show that a property $\ph(f,\delta)$ holds for
sufficiently many $\delta$ by decomposing $f$ into $E(f \mid
\YY_\delta)$ and $f - E(f \mid \YY_\delta)$. We will be able do this
by finding a long sequence such that $E(f \mid \YY_\delta)$ does not
change much over its span, and then dealing with each value, in turn.
The next lemma makes this precise.

\begin{lemma}
	Suppose there is a $\theta<\omega^{p}$ such that for every
$\alpha^{\theta}$-sequence $s$ with $\alpha\geq\omega$ and every $f$ with
$\|f\|_{L^\infty}\leq B$, there are a natural number $n$ and an
$\alpha$-subsequence $t$ of $s$ such that the property
\begin{quote}
for every  $m\geq n$, $\varphi(f,m,\delta)$
\end{quote}
holds for $t$-many $\delta$.

Suppose also that $\varepsilon > 0$ is such that whenever
$\|f-f'\|_{L^2}<\varepsilon$ and $\varphi(f,m,\delta)$ holds, also
$\varphi'(f',m,\delta)$.  Let $\varphi$ be continuous in $\delta$.
Then there is a $\theta<\omega^{2p-1}$ such that for every
$\alpha^\theta$-sequence $s$ with $\alpha\geq\omega$ and every $f$
with $\|f\|_{L^\infty}\leq B$, there are a natural number $n$ and an
$\alpha$-subsequence $t$ of $s$ such that the property
\begin{quote}
for every  $m\geq n$, $\varphi'(E(f\mid\YY_\delta),m,\delta)$ and
$\varphi'(f-E(f\mid\YY_\delta),m,\delta)$
\end{quote}
holds for $t$-many $\delta$.
\label{tech:rel}
\end{lemma}

\begin{proof}
  Give $\theta$ as in the hypothesis, we claim the conclusion holds of
  $2\theta^2+1$. If $s$ is an $\alpha^{2\theta^2+1}$-sequence, we may
  use the fact that $\alpha\geq\omega$ to divide $s$ into
  $\omega$-many $\alpha^{2\theta^2}$-sequences given by
  $s^n_\delta=s_{\alpha^{2\theta^2}\cdot n+\delta}$.  For some
  $n<\omega$,
  \[
  \|E(f\mid\YY_{s^n_0})-E(f\mid\YY_{s^n_{\alpha^{2\theta^2}}})\|<\varepsilon.
\]
  As in the previous lemma, there is an $\alpha$-subsequence $t$ of
  $s^n$ such that
\begin{quote}
for every $m\geq n$, $\varphi(E(f\mid\YY_{s^n_0}),m,\delta)$ and
$\varphi(f-E(f\mid\YY_{s^n_0}),m,\delta)$
\end{quote}
holds for $t$-many $\delta$, and the conclusion immediately follows.
\end{proof}

Our final technical lemma will give us the means to find many $\delta$ where two
properties are satisfied, where the second depends on a parameter that is chosen
to satisfy the first.

\begin{lemma}
 Suppose there is a $\theta_0<\omega^{p}$ such that for
 every $\alpha^{\theta_0}$-sequence $s$ with $\alpha\geq\omega$ and
 every $f$ with $\|f\|_{L^\infty}\leq B$, there are a natural number
 $n_0$ and an $\alpha$-subsequence $t$ of $s$ such that the property
\begin{quote}
for every $m\geq n_0$, $\varphi_0(f,m,\delta)$
\end{quote}
holds for $t$-many $\delta$.

Suppose that, additionally, for every $d$ there is a $\theta_d<\omega^{q}$
such that for every $\alpha^{\theta_d}$-sequence $s$ with
$\alpha\geq\omega$ and every $f$ with $\|f\|_{L^\infty}\leq B$, there
is a natural number $n_d$ and an $\alpha$-subsequence $t$ of $s$ with
the property
\begin{quote}
for every  $m\geq n_d$, $\varphi_d(f,m,\delta)$
\end{quote}
holds for $t$-many $\delta$.

If $\varphi_i$ is continuous in $\delta$ for each $i$ then there is a
$\theta<\omega^{p+q}$ such that for every $\alpha^\theta$-sequence $s$ with
$\alpha\geq\omega$ and every $f$ with $\|f\|_{L^\infty}\leq B$, there
are an $n$, an $N$, and an $\alpha$-subsequence $t$ of $s$ such that
the property
\begin{quote}
$\varphi_0(f,N,\delta)$ and for every $m\geq n$, $\varphi_N(f,m,\delta)$
\end{quote}
holds for $t$-many $\delta$.
\label{tech:nested}
\end{lemma}

\begin{proof}
 Let $\theta$ be $2\cdot(\sup_{d>0} \theta_d)\cdot \theta$, and let
 $s$, $f$ be given.  By the first assumption, there is an
 $\alpha^{2\cdot\sup_{d>0}\theta_d}$-subsequence of $s$, $s'$, and an
 $N$ such that $\varphi_0(f,N,\delta)$ holds for $s'$-many $\delta$.
 Then there are an $\alpha^2$-subsequence $s''$ and an $n$ such that
 both $\varphi_0(f,N,\delta)$ holds for $s''$ many $\delta$, and for
 each $m\geq n$, $\varphi_N(f,m,\delta)$ also holds for $s''$-many
 $\delta$.  Since $\alpha\geq\omega$, we may apply the method of
 Lemma \ref{tech:pair} to obtain an $\alpha$-subsequence $t$ such
 that the properties hold simultaneously for $t$-many $\delta$.
\end{proof}

Recall that if $\XX$ is a measure-preserving system and $\YY$ is a
factor, $\XX \times_\YY \XX$ is again a measure-preserving system with
factor $\YY$. The space $L^2(\XX \times_\YY \XX)$ and some of its
properties were described in Section~\ref{preliminaries:section}. The
operation of taking the relative square over $\YY$ can be iterated:
for each $r$ and $\delta$, we define the space $\XX^{[r]}_\delta$ by
induction on $r$, by setting $\XX^{[0]}_\delta$ equal to $\XX$, and
$\XX^{[r+1]}_\delta$ equal to $\XX^{[r]}_\delta\times_{\YY_\delta}
\XX^{[r]}_\delta$.

Each space $L^2(\XX^{[r]}_\delta)$ can be represented as described in
Section~\ref{preliminaries:section}. In particular,
$L^\infty(\YY_\delta)$ can be identified as a subset of
$L^2(\XX^{[r]}_\delta)$, and if $f$ and $g$ are elements of
$L^\infty(\XX^{[r]}_\delta)$, then $f \otimes g$ is an element of
$L^\infty(\XX^{[r+1]}_\delta)$. Thus the most basic elements of
$L^\infty(\XX^{[r]}_\delta)$ can be viewed as $2^r$-fold tensor
products of elements of $L^\infty(\XX)$. We define the \emph{simple}
elements of $L^\infty(\XX^{[r]}_\delta)$ to be those that can be
represented as finite sums of such basic elements.

The advantage to focusing on simple elements is that if $f$ is such an
element, then $f$ can be viewed as an element of
$L^\infty(\XX^{[r]}_\delta)$ for each $\delta$, simultaneously. More
precisely, for each $r$, we define $L_0^\infty(r)$ to be the set of
finite formal sums of such basic elements; then each element $f$ of
$L_0^\infty(r)$ denotes an element of $L^\infty(\XX^{[r]}_\delta)$,
for each $\delta$. Note that if $f$ and $g$ are elements of
$L_0^\infty(r)$ and $h$ is an element of $L^\infty(\YY)$, it makes
sense to talk about $f + g$, $h f$, and $E(f \mid \YY)$ as elements of
$L_0^\infty(r)$.  We may define the $L^\infty$ bound of such a formal
sum in the natural way, taking $\|\sum_{i<n}c_if_i\|_{L^\infty}$ to be
$\sum_{i<n}|c_i|\cdot \|f_i\|_{L^\infty}$.  Such a bound is an upper
bound for the true $L^\infty$ bound in $L^\infty(\XX^{[r]}_\delta)$
for any $\delta$, and respects the usual properties of the $L^\infty$
norm with respect to sums and products.

The next lemma shows that for each $r$, we can find many many $\delta$ such that
the space $\XX^{[r]}_\delta$ looks sufficiently weak mixing. 

\begin{lemma}
 For every $\varepsilon>0$, $B>0$, and $r$ there is a $K<\omega$ such that for
 every $\alpha^K$-sequence $s$ with $\alpha\geq\omega$ and every
 $f,g\in L_0^\infty(r)$ with $\|f\|_{L^\infty}\leq
 B,\|g\|_{L^\infty}\leq B$, there are a natural number $n$ and an
 $\alpha$-subsequence $t$ of $s$ such that the property
\begin{quote}
for every $m\geq n$,
\[\frac{1}{m}\sum_{i<m}
 \int\left[E(fT^ig \mid \YY_\delta)-E(f \mid \YY_\delta)T^iE(g \mid
   \YY_\delta) \right]^2 d\mu(\XX^{[r]}_\delta)<\varepsilon\]
\end{quote}
holds for $t$-many $\delta$.
\label{wmdim}
\end{lemma}
\begin{proof}
 By induction on $r$.  When $r=0$, this is simply Lemma
 \ref{wm:theorem}.  Suppose the claim holds for $r$.  It suffices to
 consider the case where $f$ and $g$ in $L_0^\infty(r+1)$ are of the form
$f=f_1\otimes f_2$ and $g=g_1\otimes g_2$, with $f_1, f_2, g_1, g_2$ in
$L_0^\infty(r)$.
 Using Lemma \ref{tech:rel} and the subadditivity of the left hand
 side, it suffices to consider the cases where $E(f_i\mid\YY_\delta)=0$ and
 where $E(f_i\mid\YY_\delta)=f_i$; the case where
 $E(f_i\mid\YY_\delta)=f_i$ for both $i=1$ and $i=2$ is trivial, so
 we may further assume that for some $i\in\{1,2\}$,
 $E(f_i\mid\YY_\delta)=0$.  

%(Note that we take advantage of the fact that Lemma \ref{wm:theorem} holds with
%$K<\omega$; if the bound in that lemma were larger, the bound here might be
%quite complicated, and in particular, depend on the complexity of $f$ and $g$.)

 By the inductive hypothesis and Lemma~\ref{tech:pair}, for any
 $\varepsilon' > 0$ we can find $K$ large enough so that every
 $\alpha^K$-sequence $s$ has an $\alpha$-subsequence $t$ such that
\[\frac{1}{m}\sum_{i<m}
 \int\left[
E(f_1T^ig_1\mid\YY_\delta)-E(f_1\mid\YY_\delta)E(T^ig_1\mid\YY_\delta)\right]^2
d\mu(\XX^{[r]}_\delta)<\varepsilon'\]
and
\[\frac{1}{m}\sum_{i<m}
 \int\left[
E(f_2T^ig_2\mid\YY_\delta)-E(f_2\mid\YY_\delta)E(T^ig_2\mid\YY_\delta)\right]^2
d\mu(\XX^{[r]}_\delta)<\varepsilon'\]
for $t$-many $\delta$.  But then, for such $\delta$,
\begin{multline*}
\frac{1}{m}\sum_{i<m}
 \int\left[E((f_1\otimes f_2)(T^ig_1\otimes T^ig_2) \mid \YY_\delta)
 \right]^2 d\mu(\XX^{[r+1]}_\delta)= \\
\frac{1}{m}\sum_{i<m}
 \int\left[E(f_1T^ig_1\mid\YY_\delta)E(f_2T^ig_2\mid\YY_\delta)\right]^2
d\mu(\XX^{[r]}_\delta)
\end{multline*}
is close to
\[
\frac{1}{m}\sum_{i<m}
\int\left[
E(f_1\mid\YY_\delta)T^iE(g_1\mid\YY_\delta)E(f_2\mid\YY_\delta)T^i
E(g_2\mid\YY_\delta)\right]^2
d\mu(\XX^{[r]}_\delta),
\] which is $0$ since either
$E(f_1\mid\YY_\delta)=0$ or $E(f_1\mid\YY_\delta)=0$.
\end{proof}

From this point on, our proof follows that of \cite[Theorem
8.3]{furstenberg:et:al:82} very closely.

\begin{lemma}
 Suppose that for every $\varepsilon>0$, $B>0$, $k$, and $r$ there is a
 $\theta<\omega^{p}$ such that for every $\alpha^\theta$-sequence $s$
 with $\alpha\geq\omega$ and every $f_0,\ldots,f_{k-1}$ in
 $L_0^\infty(r)$ with $\|f_i\|_{L^\infty}\leq B$, there are a natural
 number $n$ and an $\alpha$-subsequence $t$ of $s$ such that the
 property
\begin{quote}
for every $m\geq n$,
\[\frac{1}{m}\sum_{i<m}
 \int\left(E(\prod_{l=0}^{k-1}
T^{li}f_l\mid\YY_\delta)-\prod_{l=0}^{k-1}
T^{li}E(f_l\mid\YY_\delta)\right)^2
d\mu(\XX^{[r]}_\delta)<\varepsilon\]
\end{quote}
holds for $t$-many $\delta$.

 Then for every $\varepsilon>0,B>0,k,r$ there is a $\theta<\omega^{kp+k-1}$
such that for every $\alpha^\theta$-sequence $s$ with $\alpha\geq\omega$
and every $f_1,\ldots,f_k$ with $\|f_i\|_{L^\infty}\leq B$, there are
a natural number $n$ and an $\alpha$-subsequence $t$ of $s$ such that
the property
\begin{quote}
for every $m\geq n$,
\[\|\frac{1}{m}\sum_{i<m}\left(\prod_{l=1}^k T^{li}f_l-\prod_{l=1}^k
T^{li}E(f_l\mid\YY_\delta)\right)\|_{L^2(\XX^{[r]}_\delta)}<\varepsilon\]
\end{quote}
holds for $t$-many $\delta$.
\label{wmstrong}
\end{lemma}
\begin{proof}
Under the additional assumption that for some $l_0$,
$E(f_{l_0}\mid\YY_\delta)=0$, we will prove the claim with
$\theta<\omega^{p+1}$.  Since
\[\prod_{l=1}^k T^{li}f_l-\prod_{l=1}^k
T^{li}E(f_l\mid\YY_\delta)=\sum_{j=1}^k\left(\prod_{l=1}^{j-1}T^{li}f_l\right)T^
{ji}\left(f_j-E(f_j\mid\YY_\delta)\right)\left(\prod_{j+1}^k
T^{li}E(f_l\mid\YY_\delta)\right),\]
we will then be able to apply Lemma \ref{tech:pair} $k-1$ times to
obtain the full result with the stated bound.

So assume that $E(f_{l_0}\mid\YY_\delta)=0$.  By Lemma \ref{wmdim},
Lemma \ref{tech:nested}, and the assumption, we may choose a
$\theta<\omega^{p+1}$ so that for every $\alpha^\theta$-sequence $s$
and every $f_1,\ldots,f_k$ with $\|f_i\|_{L^\infty}\leq 1$, there are
natural numbers $N$ and $H$ and an $\alpha$-subsequence $t$ of $s$
such that for some $\varepsilon > 0$, chosen small enough for the argument
below, the property
\begin{quote}
\[
\frac{1}{H}\sum_{r=1-H}^{H-1}\int\left[E(f_{l_0}T^{l_0r}f_{l_0}
\mid\YY_\delta)-E(f_{l_0}\mid\YY_\delta)T^{l_0r}E(f_{l_0}\mid\YY_\delta)\right]
^2d\mu(\XX^{[r]}_\delta)<\varepsilon/k
\]
and for every $m\geq N$ and $|r|<H$,
\[
\frac{1}{m}\sum_{i<m}\int\left[E(\prod_{l=1}^kT^{(l-1)i}f_lT^{lr}
f_l\mid\YY_\delta)-\prod_{l=1}^kT^{(l-1)i}E(f_lT^{lr}f_l\mid\YY_\delta)\right]
^2d\mu(\XX^{[r]}_\delta)<\varepsilon/k
\]
\end{quote}
holds for $t$-many $\delta$.  It will suffice to argue that these two
properties, at any $\delta$, imply that for some $n$,
\[
\left\| \frac{1}{m}\sum_{i<m}\left(\prod_{l=1}^k T^{li}f_l-\prod_{l=1}^k
T^{li}E(f_l\mid\YY_\delta)\right) \right\|_{L^2(\XX^{[r]}_\delta)}<\varepsilon.
\]

The necessary $n$ is $\max\{N,cH\}$ for some large constant $c$
depending on $\varepsilon$.  Let $m\geq n$ be given.  Then, since $m$ is
much larger than $H$, it suffices to show that the properties above
imply
\[\left\|\frac{1}{m}\sum_{i<m}\frac{1}{H}\sum_{h=i}^{i+H-1}\prod_{l=1}^k
T^{lh}f_l\right\|\]
is small.  By the convexity of $x^2$, it suffices to show that
\[\frac{1}{m}\sum_{i<m}\int\left(\frac{1}{H}\sum_{h=i}^{i+H-1}\prod_{l=1}^k
T^{lh}f_l\right)^2d\mu(\XX^{[r]}_\delta)\]
is small.  Expanding, this is bounded by
\[\frac{1}{m}\sum_{i<m}\frac{1}{H^2}\sum_{h,h'=i}^{i+H-1}\int
\prod_{l=1}^k T^{lh}f_lT^{lh'}f_l d\mu(\XX^{[r]}_\delta).\]
But this may be rewritten as
\[\frac{1}{H}\sum_{r=1-H}^{H-1}\left(1-\frac{|r|}{H}\right)\left[\frac{1}{m}
\sum_{i<m}\int\prod_{l=1}^k
T^{(l-1)i}(f_lT^{lr}f_l)\right] d\mu(\XX^{[r]}_\delta).\]
Since we have chosen $m\geq N$, this is close to
\[\frac{1}{H}\sum_{r=1-H}^{H-1}\left(1-\frac{|r|}{H}\right)\left[\frac{1}{m}
\sum_{i<m}\int\prod_{l=1}^k
T^{(l-1)i}E(f_lT^{lr}f_l\mid\YY_\delta)
d\mu(\XX^{[r]}_\delta)\right]\]
which is bounded by
\[\frac{1}{H}\sum_{r=1-H}^{H-1} \left(1-\frac{|r|}{H}\right)
\|E(f_{l_0}T^{l_0r}f_{l_0}\mid\YY_\delta)\|_{L^2(\XX^{[r]}_\delta)}\prod_{l\neq
l_0}\|f_l\|^2_{\infty}.\]
But we have chosen $H$ large enough that
$\|E(f_{l_0}T^{l_0r}f_{l_0}\mid\YY_\delta)\|$ is close to $0$ for
almost every $r$, and since the terms are bounded by $\prod_l
\|f_l\|_{\infty}^2$, the average is small as well.
\end{proof}

\begin{lemma}
     Suppose that for every $\varepsilon>0,B>0,q,k$, and $r$, there is a
$\theta<\omega^{p}$ such that for every $\alpha^\theta$-sequence $s$ with
$\alpha\geq\omega$ and every $f_1,\ldots,f_k$ in
$L_0^\infty(2^{r+1})$ with $\|f_l\|_{L^\infty}\leq B$ for each
$l\leq k$, there are a natural number $n$ and an $\alpha$-subsequence
$t$ of $s$ such that the property
\begin{quote}
for every $m\geq n$,
\[\left\|\frac{1}{m}\sum_{i<m}\left(\prod_{l=1}^k T^{li}f_l-\prod_{l=1}^k
T^{li}E(f_l\mid\YY_\delta)\right)\right\|_{L^2(\XX^{[r+1]}_\delta)}<\varepsilon
\]
\end{quote}
holds for $t$-many $\delta$.

Further, suppose that for every $\varepsilon>0, B>0,q,k$ and $r$,
there is a $\theta<\omega^{q}$ such that for every
$\alpha^\theta$-sequence $s$ with $\alpha\geq\omega$ and every
$f_0,\ldots,f_{k-1}$ in $L_0^\infty(2^{r})$ with
$\|f_l\|_{L^\infty}\leq B$ for each $l\leq k$, there are a natural
number $n$ and an $\alpha$-subsequence $t$ of $s$ such that the
property
\begin{quote}
for every $m\geq n$,
\[\frac{1}{m}\sum_{i<m}
 \int\left(E(\prod_{l=0}^{k-1}
T^{li}f_l\mid\YY_\delta)-\prod_{l=0}^{k-1}
T^{li}E(f_l\mid\YY_\delta)\right)^2
d\mu(\XX^{[r]}_\delta)<\varepsilon\]
\end{quote}
holds for $t$-many $\delta$.

Then for every $\varepsilon>0,B>0,q,k$, and $r$, there is a
$\theta<\omega^{p+q-1}$ such that for every $\alpha^\theta$-sequence
$s$ with $\alpha\geq\omega$ and every $f_0,\ldots,f_k$ in
$L_0^\infty(2^{r})$ with $\|f_l\|_{L^\infty}\leq B$ for each $l\leq
k$, there are a natural number $n$ and an $\alpha$-subsequence $t$ of
$s$ such that the property
\begin{quote}
for every $m\geq n$,
\[\frac{1}{m}\sum_{i<m}
 \int\left(E(\prod_{l=0}^k T^{li}f_l\mid\YY_\delta)-\prod_{l=0}^k
T^{li}E(f_l\mid\YY_\delta)\right)^2
d\mu(\XX^{[r]}_\delta)<\varepsilon\]
\end{quote}
holds for $t$-many $\delta$.
\label{wmweak}
\end{lemma}
\begin{proof}
 Once again, we apply Lemma \ref{tech:rel} and subadditivity to
reduce to the two cases where $E(f_0\mid\YY_\delta)=0$ and where
$E(f_0\mid\YY_\delta)=f_0$.

In the former case, we may use the first hypothesis to choose
witnesses so that
\[\left\|\frac{1}{m}\sum_{i<m}\left(\prod_{l=1}^k T^{li}f_l-\prod_{l=1}^k
T^{li}E(f_l\mid\YY_\delta)\right)\right\|_{L^2(\XX^{[r+1]}_\delta)}
<\varepsilon/2.\]
Then it suffices to show
\[\int f_0\otimes f_0 \frac{1}{m}\sum_{i<m} \prod_{l=1}^k
T^{li}(f_l\otimes f_l) d\mu(\XX^{[r+1]}_\delta)<\varepsilon.\]
But by the choice of witnesses, the left hand side is within $\varepsilon$ of
\[\int f_0\otimes f_0 \frac{1}{m}\sum_{i<m}\prod_{l=1}^k
T^{li}E(f_l\mid\YY_\delta) d\mu(\XX^{[r+1]}_\delta)\]
and since $E(\frac{1}{m}\sum_{i<m}\prod_{l=1}^k
T^{li}E(f_l\mid\YY_\delta)\mid\YY_\delta)=\frac{1}{m}\sum_{i<m}\prod_{l=1}^k
T^{li}E(f_l\mid\YY_\delta)$ and $E(f_0\mid\YY_\delta)=0$, it follows
that this expression is $0$.

In the latter case, we may use the second hypothesis to choose witnesses so that
\[\frac{1}{m}\sum_{i<m}
 \int\left(E(\prod_{l=0}^{k-1}
T^{li}f_{l+1}\mid\YY_\delta)-\prod_{l=0}^{k-1}
T^{li}E(f_{l+1}\mid\YY_\delta)\right)^2
d\mu(\XX^{[r]}_\delta)<\varepsilon.\]
Then the left hand side of the desired conclusion is bounded by
\[\|f_0\|_{L^\infty}^2\frac{1}{m}\sum_{i<m}\int
\left(E(\prod_{l=1}^{k} T^{li}f_{l}\mid\YY_\delta)-\prod_{l=1}^{k}
T^{li}E(f_{l}\mid\YY_\delta)\right)^2 d\mu(\XX^{[r]}_\delta)\]
and shifting each term by $T^{-li}$, this is equal to
\[\|f_0\|_{L^\infty}^2\frac{1}{m}\sum_{i<m}\int
\left(E(\prod_{l=0}^{k-1}
T^{li}f_{l+1}\mid\YY_\delta)-\prod_{l=0}^{k-1}
T^{li}E(f_{l+1}\mid\YY_\delta)\right)^2 d\mu(\XX^{[r]}_\delta)\]
which is less than $\varepsilon$.
\end{proof}

\begin{lemma}
\label{final:lemma}
\begin{enumerate}
\item \label{conc:strong} For every $\varepsilon>0,B>0$, and $k$,
  there is a $\theta<\omega^{k^{2k}}$ such that for every
  $\alpha^\theta$-sequence $s$ with $\alpha\geq\omega$ and every
  $f_0,\ldots,f_{k}$ in $L^\infty(\XX)$ with $\|f_l\|_{L^\infty}\leq
  B$ for each $l\leq k$, there are a natural number $n$ and an
  $\alpha$-subsequence $t$ of $s$ such that the property
\begin{quote}
for every $m\geq n$,
\[\frac{1}{m}\sum_{i<m}
 \int\left(E(\prod_{l=0}^{k} T^{li}f_l\mid\YY_\delta)-\prod_{l=0}^{k}
T^{li}E(f_l\mid\YY_\delta)\right)^2 d\mu(\XX)<\varepsilon\]
\end{quote}
holds for $t$-many $\delta$.
\item For every $\varepsilon>0,B>0$, and $k$, there is a
  $\theta<\omega^{k^{2k-1}}$ such that for every
  $\alpha^\theta$-sequence $s$ with $\alpha\geq\omega$ and every
  $f_1,\ldots,f_{k}\in L^\infty(\XX^{2^{r}})$ with
  $\|f_l\|_{L^\infty}\leq B$ for each $l\leq k$, there are a natural
  number $n$ and an $\alpha$-subsequence $t$ of $s$ such that the
  property
\begin{quote}
for every $m\geq n$,
\[\left\|
 \frac{1}{m}\sum_{i<m}
 \left(\prod_{l=1}^k T^{li} f_l - \prod_{l=1}^k T^{li} E(f_l \mid
   \YY_\delta) \right) \right\|_{L^2(\XX)} < \varepsilon\]
\end{quote}
holds for $t$-many $\delta$.
\label{conc:weak}
\end{enumerate}
\end{lemma}
\begin{proof}
 We will prove the stronger claim that these hold with any
 $\XX^{[r]}_\delta$ in place of $\XX$ and $L_0^\infty(r)$ in place of
 $L^\infty(\XX)$, simultaneously by induction on $k$.  For $k=1$,
 (\ref{conc:strong}) is Lemma \ref{wmdim} and (\ref{conc:weak}) is
 trivial.  Given (\ref{conc:strong}) for $k$, (\ref{conc:weak}) for
 $k+1$ follows by Lemma \ref{wmstrong}.  Given (\ref{conc:weak}) for
 $k+1$ and (\ref{conc:strong}) for $k$, (\ref{conc:strong}) for $k+1$
 follows by Lemma \ref{wmweak}.
\end{proof}

Theorem~\ref{main:theorem} and Corollary~\ref{main:corollary} follow
by taking $s$ to be the $\alpha^\theta$-sequence with $s_\beta =
\beta$ for every for every $\beta \leq \alpha^\theta$. 

% The bounds in the above lemma are certainly not optimal, but are much
% simpler than those given by precisely solving the recurrence relation
% needed to find the exact bounds.

\section{Logical issues}
\label{logic:section}

We now turn to a discussion of the logical methods behind the results
just obtained. This paper is part of a broader to effort to understand
the methods of ergodic theory and ergodic Ramsey theory in more
explicit computational or combinatorial terms \cite{avigad:unp:n},
using a body of logical techniques that fall under the heading ``proof
mining'' (see \cite{kohlenbach:08,kohlenbach:oliva:03a}, as well as
\cite[Section 6]{avigad:et:al:unp}). In particular, the results here
were obtained by employing a systematic rewriting of the
Furstenberg-Katznelson proof
\cite{furstenberg:katznelson:79,furstenberg:81,furstenberg:et:al:82},
based on G\"odel's \emph{Dialectica} functional interpretation
\cite{goedel:58,avigad:feferman:98}. Here we provide a ``rational
reconstruction'' of the methods we used.

The first step was to rewrite the key definitions and lemmas in the
Furstenberg-Katznelson proof in a way that makes the logical structure
of the assertions clear, and, in particular, distinguishes
quantification over ordinals from quantification over integers and
other objects that have a finitary representation. Limits and
projections involving the maximal distal factor, $\YY$, were expressed
directly in terms of the hierarchy $(\YY_\alpha)$. For example, the
assertion that the projection $E(f \mid \YY)$ is within $\varepsilon$
of $g$ can be expressed as $\exists \alpha \; \forall \beta > \alpha
\; \| E(f \mid \YY_\beta) - g \| \leq \varepsilon$, which asserts that
there is a level $\alpha$ beyond which the projection stays within
$\varepsilon$ of $g$. But it can also be expressed as $\forall \alpha
\; \exists \beta > \alpha \; (\| E(f \mid \YY_\beta) - g \| \leq
\varepsilon)$, which asserts that there are arbitrarily large levels
$\beta$ at which the projection is within $\varepsilon$ of $g$. The
statement that the sequence $A_n (f \otimes f)$ converges in $\XX
\times_\YY \XX$ can then be expressed as follows:
\begin{equation}
\label{erg:eq:a}
\forall \varepsilon > 0 \; \exists n \; \forall m \geq n, \alpha \;
\exists \beta > \alpha \; (\| A_m (f \otimes f) - A_n (f \otimes f)
\|_{L^2(\XX \times_{\YY_\beta} \XX)} < \varepsilon).
\end{equation}
Other statements central to the proof were analyzed in similar
ways. 

The proof of the mean ergodic theorem is not constructive
\cite{avigad:et:al:unp,avigad:unp:n}, and, in general, once cannot
extract bounds on $\beta$ in (\ref{erg:eq:a}). The next step was
therefore to seek a ``quasi-constructive'' interpretation of the proof
which yields more explicit ordinal bounds. To that end, we employed a
functional interpretation roughly along the lines of the one described
in \cite{avigad:towsner:unp} (which is, in turn, related to a similar
interpretation due to Feferman, 
described in \cite[Section 9.3]{avigad:feferman:98}). For example, in
(\ref{erg:eq:a}), the dependence of $\beta$ on $m$ can be eliminated
by choosing a $\beta_m$ for each $m$, and then taking the supremum:
\begin{equation*}
% \label{erg:eq:b}
 \forall \varepsilon > 0 \; \exists n \; \forall \alpha \;
 \exists \beta \; (\beta > \alpha \land \forall m \geq n \; 
 \| A_m (f \otimes f) - A_n (f \otimes f) \|_{L^2(\XX \times_{\YY_\beta} \XX)} <
\varepsilon).
\end{equation*}
 We can then make the dependence of $\beta$ on $\alpha$ explicit:
\begin{multline}
\label{erg:eq:c}
\forall \varepsilon > 0 \; \exists n, \beta \; \forall \alpha \;
(\beta(\alpha) > \alpha \mathop{\land} \\
\forall m \geq n \; \| A_m (f \otimes f) - A_n (f \otimes f)
\|_{L^2(\XX \times_{\YY_{\beta(\alpha)}} \XX)} < \varepsilon).
\end{multline}
It is still impossible to obtain an explicit description of $\beta$,
but the Dialectica interpretation involves one final move. If
(\ref{erg:eq:c}) were false, then for some fixed $\varepsilon > 0$,
there would be a function $\alpha(n,\beta)$ that provided a
counterexample for each $n$ and $\beta$. Thus (\ref{erg:eq:c}) is
equivalent to the assertion that there is no such counterexample:
\begin{multline}
\label{erg:eq:d}
\forall \varepsilon > 0, \alpha \; \exists n, \beta \; 
(\beta(\alpha(n,\beta)) > \alpha(n,\beta) \mathop{\land} \\ 
\forall m \geq n \; 
\| A_m (f \otimes f) - A_n (f \otimes f) \|_{L^2(\XX
 \times_{\YY_{\beta(\alpha(m,\beta))}} \XX)} < \varepsilon).
\end{multline}
The logical methods now make it possible to extract an explicit
description of the function $\beta$ that ``foils'' the purported
counterexample $\alpha$. Informally, one obtains an algorithm for
$\beta$ which involves relatively explicit operations with ordinals,
such as taking maxima and suprema; application and iterations of
functions; and possibly noncomputable functions on the integers. (The
fact that transfinite induction is not used in the proof of the mean
ergodic theorem for $\XX \times_\YY \XX$ translates to the fact that
there are no transfinite recursions in the algorithm. Allowing noncomputable
functions on the integers allows us to ignore, for example, the universal 
quantifier over $m$ in (\ref{erg:eq:d}), and restrict focus to the parts of the
informal proof that bear on the ordinal bounds.) More
formally, one obtains a term in the calculus denoted
$\mathit{T_\Omega}$ in \cite{avigad:towsner:unp}, involving only the
operations just mentioned.

In the final result, Theorem~\ref{main:theorem}, there is only an
existential quantifier over ordinals. Methods of Tait \cite{tait:65}
(see also \cite[Section 4.4]{avigad:feferman:98}) suggest that the
explicit witnessing term extracted from the proof should be bounded
below the ordinal $\varepsilon_0$, which is the limit of the ordinals
$\omega^\omega, \omega^{\omega^\omega}, \ldots$. The final step of our
analysis was to seek a more direct route to obtain such a conclusion,
both to improve the bound and avoid relying on metamathematical
considerations. For example, if one is interested in bounds rather
than explicit witnesses in (\ref{erg:eq:d}), one can assume that the
function $\beta$ is increasing and continuous. Given any such
function, $\beta$, there are unboundedly many ordinals $\gamma$ that
are closed under $\beta$. Inspection of the translated proof of
(\ref{erg:eq:d}) showed that it was possible to think of the
counterexample function, $\alpha$, as taking such a sequence of
closure ordinals, and returning a sequence of bounds on
counterexamples; the proof showed that the original sequence could be
thinned to obtain a subsequence along which $\alpha$ fails. Once the
decision was made to cast the central results in those terms, it was
fairly easy to describe the algorithms extracted by the functional
interpretation in that way.

 The analysis yields not only the additional information provided by
 Theorem~\ref{main:theorem}, but also shows that the argument does
 not use the full axiomatic strength needed to carry out the
 transfinite iteration. The transfinite construction of the
 Furstenberg-Zimmer structure theorem requires an impredicative
 theory, like $\mathit{ID_1}$ or $\mathit{\Pi^1_1\mathord{-}CA}$,
 which is, from a proof-theoretic standpoint, quite strong; in
 contrast, the construction of the hierarchy up to stage
 $\omega^{\omega^\omega}$ requires only a principle of iterated
 arithmetic comprehension along that ordinal, which can be obtained,
 for example, in the predicative theory
 $\mathit{\Sigma^1_1\mathord{-}CA}$. See
 \cite{avigad:unp:n,avigad:feferman:98,simpson:99} for more
 information about the relevant theories.

 It is interesting to note, however, that the logical considerations
 drop out of the final results. The metamathematical results provide
 a deeper understanding of the role that strong nonconstructive
 principles play in ordinary mathematical reasoning, and provide a
 guide to interpreting particular mathematical proofs in more
 explicit terms. But if one is only interested in the latter, at the
 end of the day, one is left with a purely mathematical proof.

% \bibliography{proofthry,ergodic}
% \bibliographystyle{plain}

\end{document}